\documentclass[10pt]{amsart}
\usepackage{amssymb, enumitem}
\usepackage[all]{xy}
\usepackage{tikz}
\usepackage{relsize}
\usepackage{aliascnt}
\usepackage{verbatim}
\usepackage[english]{babel}
\usepackage{enumitem}
\usepackage{mathtools}
\usepackage[title]{appendix}
\usepackage{amsmath} 
\usepackage{amssymb} 
\usepackage{mathrsfs}  
\usepackage{mathtools} 
\usepackage[T1]{fontenc} 
\usepackage[allcolors=black]{hyperref} 
\usepackage[utf8]{inputenc} 
\usepackage{bm} 

\usepackage{graphicx} 
\usepackage{tikz-cd} 
\usepackage[noadjust]{cite} 
\usepackage{todonotes} 
\usepackage{setspace} 

\usepackage{amsthm} 

\def\today{\number\day\space\ifcase\month\or   January\or February\or
	March\or April\or May\or June\or   July\or August\or September\or
	October\or November\or December\fi\   \number\year}

\theoremstyle{plain}
\newtheorem{lma}{Lemma}[section]

\newaliascnt{thmCt}{lma}
\newtheorem{thm}[thmCt]{Theorem}
\aliascntresetthe{thmCt}

\newaliascnt{corCt}{lma}
\newtheorem{cor}[corCt]{Corollary}
\aliascntresetthe{corCt}

\newaliascnt{propCt}{lma}

\aliascntresetthe{propCt}

\newtheorem*{thm*}{Theorem}
\newtheorem*{cor*}{Corollary}
\newtheorem*{prop*}{Proposition}
\newtheorem*{lma*}{Lemma}

\theoremstyle{plain}
\newtheorem{thmintro}{Theorem} 

\newaliascnt{introcorCt}{thmintro}
\newtheorem{introcor}[introcorCt]{Corollary}
\aliascntresetthe{introcorCt}

\newaliascnt{egintroCt}{thmintro}

\newaliascnt{propintroCt}{thmintro}
\newtheorem{propintro}[propintroCt]{Proposition}

\theoremstyle{definition}

\newaliascnt{introquesCt}{thmintro}

\aliascntresetthe{introquesCt}

\newaliascnt{pbmintroCt}{thmintro}

\aliascntresetthe{introquesCt}

\newaliascnt{dfnintroCt}{thmintro}

\aliascntresetthe{introquesCt}

\newaliascnt{pgrCt}{lma}

\aliascntresetthe{pgrCt}

\newaliascnt{dfCt}{lma}

\aliascntresetthe{dfCt}

\newaliascnt{remCt}{lma}
\newtheorem{rem}[remCt]{Remark}
\aliascntresetthe{remCt}

\newaliascnt{remsCt}{lma}

\aliascntresetthe{remsCt}

\newaliascnt{egCt}{lma}

\aliascntresetthe{egCt}

\newaliascnt{egsCt}{lma}

\aliascntresetthe{egsCt}

\newaliascnt{qstCt}{lma}

\aliascntresetthe{qstCt}

\newaliascnt{pbmCt}{lma}

\aliascntresetthe{pbmCt}

\newaliascnt{notaCt}{lma}

\aliascntresetthe{notaCt}

\theoremstyle{theorem}

\newcommand{\N}{\mathbb N}

\newcommand{\Z}{\mathbb Z}

\newcommand{\F}{\mathbb F}

\DeclareMathOperator{\Hom}{Hom}

\DeclareMathOperator{\pt}{pt}

\DeclareMathOperator{\Cl}{\mathrm{Cl}}

\author[S. Geffen]{Shirly Geffen}
\address{Shirly Geffen, Mathematisches Institut, Fachbereich Mathematik und Informatik der
	Universit\"at M\"unster, Einsteinstrasse 62, 48149 M\"unster, Germany.
}
\email{sgeffen@uni-muenster.de}
\urladdr{https://shirlygeffen.com/}

\author[J. Kranz]{Julian Kranz}
\address{Julian Kranz, Institute for Analysis and Numerics,
	Applied Mathematics M\"unster,
	Faculty of Mathematics and Computer Science,
	University of M\"unster,
	Einsteinstrasse 62,
	48149 M\"unster,
	Germany.}
\email{julian.kranz@uni-muenster.de}
\urladdr{https://sites.google.com/view/juliankranz/}

\thanks{
	J. Kranz was supported by the Engineering and Physical Sciences Research Council [Grant Ref: EP/X026647/1].
	S. Geffen was partially supported by the ERC Advanced Grant 834267 - AMAREC, by the Deutsche Forschungsgemeinschaft (DFG, German Research Foundation) - Project-ID 427320536 - SFB 1442, as well as by Germany's Excellence Strategy EXC	2044-390685587, Mathematics M\"unster: Dynamics-Geometry-Structure.
}

\title[$C^*$-algebras associated to boundary actions of 3-manifold groups]{Note on $C^*$-algebras associated to boundary actions of hyperbolic $3$-manifold groups}

\begin{document}
	
	\begin{abstract}
		Using Kirchberg--Phillips' classification of purely infinite $C^*$-al\-ge\-bras by $K$-theory, we prove that the isomorphism types of crossed product $C^*$-algebras associated to certain hyperbolic $3$-manifold groups acting on their Gromov boundary only depend on the manifold's homology. As a result, we obtain infinitely many pairwise non-isomorphic hyperbolic groups all of whose associated crossed products are isomorphic. These isomomorphisms are not of dynamical nature in the sense that they are not induced by isomorphisms of the underlying groupoids. 
	\end{abstract}
	\maketitle
	
	\section{Introduction}
	
	The Kirchberg--Phillips classification theorem \cite{Kirchberg, Phillips} classifies so-called \emph{UCT Kirchberg $C^*$-algebras} up to isomorphism by topological $K$-theory. As a consequence, this theorem produces many `surprising' isomorphisms of $C^*$-algebras that cannot be constructed by hand, such as Kirchberg's isomorphism $\mathcal O_2 \otimes \mathcal O_2 \cong \mathcal O_2$. The class of UCT Kirchberg $C^*$-algebras contains many examples such as the (reduced) crossed product $C^*$-algebra $C(\partial G)\rtimes_r G$ associated to the action of a non-elementary torsion-free hyperbolic group $G$ on its Gromov-boundary \cite{claire, Tu}. In fact, crossed products of arbitrary minimal amenable topologically free actions $G\curvearrowright X$ of acylindrically hyperbolic groups belong to this class \cite{Classifiability}.
	Although this large class of $C^*$-algebras is \emph{in principle} classifiable by $K$-theory, explicit $K$-theory computations remain difficult in general. 
	For the specific example of the action of a torsion-free hyperbolic group $G$ on its boundary, Emerson and Meyer \cite{Gysin} relate the $K$-theory of $C(\partial G)\rtimes_r G$ to the $K$-theory and $K$-homology of the classifying space $BG$. 
	In this short note, we use their result to obtain explicit computations for fundamental groups of hyperbolic $3$-manifolds.
	Our techniques are identical to an analogous result for non-compact manifolds \cite[Proposition~1.6]{MM}

	\begin{thmintro}\label{thm-intro}
		Let $d\geq 0$ be a fixed integer and let $G$ be the fundamental group of a closed connected orientable hyperbolic $3$-manifold $M$ with $H_1(M)\cong \Z^d$. Then there are isomorphisms
		\[K_0(C(\partial G)\rtimes_r G)\cong K_1(C(\partial G)\rtimes_r G) \cong \Z^{2d+2}\]
		that identify the class of the unit $[1]\in K_0(C(\partial G)\rtimes_r G)$ with $(1,0)\in \Z\oplus \Z^{2d+1}$. 
	\end{thmintro}
	Brock--Dunfield's work \cite{brockdunfield} shows that there exists a rich supply of groups that satisfy the assumptions in \autoref{thm-intro}. Thus, we obtain examples of isomorphisms of crossed product $C^*$-algebras:
	\begin{introcor}\label{cor-intro}
		There are infinitely many pairwise non-isomorphic hyperbolic groups $G_i,~ i\in \N$ with pairwise isomorphic crossed products $C(\partial G_i)\rtimes_r G_i,~ i\in \N$.
	\end{introcor}
	
	These isomorphisms are `surprising' in the sense that they cannot be obtained from an isomorphism of the underlying groupoids.
	
	\begin{propintro}\label{prop-intro}
		Let $G_i$ and $G_j$ be two non-isomorphic groups as in \autoref{thm-intro}. Then there is no isomorphism $C(\partial G_i)\rtimes_r G_i\cong C(\partial G_j)\rtimes_r G_j$ that restricts to an isomorphism $C(\partial G_i)\cong C(\partial G_j)$. In particular, $C(\partial G_i)\rtimes_r G_i$ contains infinitely many pairwise non-conjugate Cartan subalgebras with spectrum $S^2$. 
	\end{propintro}

		\subsection*{Acknowledgments} 
	This note arose from discussions with Johannes Ebert and is mostly based on his ideas. We are indebted to him for sharing his strategy to combine the Baum--Connes conjecture with the Atiyah--Hirzebruch spectral sequence, for educating us about hyperbolic $3$-manifolds, for guiding us through the computations, and for encouraging us to use his ideas and publish this note. 
	After this note was circulated as a preprint, Bram Mesland and Haluk \c{S}eng\"{u}n brought to our attention that \autoref{thm-intro} can be deduced from their joint work \cite{MM}.
	We are moreover grateful to Jakub Curda for pointing out the reference \cite{Matthey} which drastically simplified the proof.

	\section{Isomorphism types of crossed products}\label{sec:K_theory}
	
	Let $M$ be a closed connected orientable hyperbolic $3$-manifold. Then the universal cover $\widetilde{M}$ of $M$ is isometric to the $3$-dimensional hyperbolic space $\mathbb{H}^3$ by the Killing--Hopf theorem \cite[Corollary~12.5]{Lee}. 
	Let $G=\pi_1(M)$ be the fundamental group of $M$. Considering the natural (free) action $G\curvearrowright \widetilde{M}$ by Deck transformations, and the fact that $\widetilde{M}$ is contractible, the orbit space $\widetilde{M}/G\cong M$ is a finite model for the classifying space $BG$. 
	It follows that $G$ is finitely generated and torsion-free \cite[\S VIII, Exercise~4.1 and Corollary~2.5]{Brown}.
	Moreover, since the action $G\curvearrowright \widetilde{M}$ is proper and co-compact, $G$ is quasi-isometric to $\mathbb{H}^3$ by the Milnor-\v{S}varc lemma \cite[\S I.8, Proposition~8.9]{bridsonhaeflinger}. In particular, $G$ is a Gromov-hyperbolic group with Gromov boundary $\partial G\cong \partial  \mathbb{H}^3\cong S^2$ \cite[\S III H, Theorem~1.9 and Theorem~3.9]{bridsonhaeflinger}.\\ 
	
	We refer the reader to \cite{K-theory} for the definition of crossed product $C^*$-algebras, topological $K$-theory for $C^*$-algebras, and the Baum--Connes conjecture. For a unital $C^*$-algebra $A$, we denote by $[1_A]\in K_0(A)$ the class of its unit (or equivalently the class of the rank-one $A$-module $[A]$). \\
	
	For a connected space $X$ and a choice of base point $x_0\in X$, we denote by $[x_0]\in H_0(X)$ the image of the canonical generator $1\in \Z=H_0(\{x_0\})$ under the natural map $H_0(\{x_0\})\to H_0(X)$ and by $[1_X]\in K_0(X)$ the image of the canonical generator $1\in \Z=K_0(\{x_0\})$ under the natural map $K_0(\{x_0\})\to K_0(X)$. 
	
	The following computations of $K$-theory and $K$-homology of $3$-manifolds using their cohomology and homology groups is an application of the Atiyah-Hirzebruch spectral sequence. The proof is given in \cite[Proposition 2.1]{Matthey}, but we include a brief introduction to spectral sequences and more detailed computations for non-experts in \autoref{appendixA}. Note that a similar result was obtained in \cite[Section~1]{MM}.
	
	\begin{lma}\label{lma-k-homology}
		Let $M$ be a closed connected orientable $3$-manifold. Then there are isomorphisms
		\begin{align}
			K_0(M)&\cong H_0(M)\oplus H_2(M)\label{iso1},\\
			K_1(M)&\cong H_1(M)\oplus H_3(M)\label{iso2},\\
			K^0(M)&\cong H^0(M)\oplus H^2(M)\label{iso3},\\
			K^1(M)&\cong H^1(M)\oplus H^3(M)\label{iso4},
		\end{align}
		that identify $[1_M]\in K_0(M)$ with $[x_0]\in H_0(M)$.
	\end{lma}
	\begin{proof}
		The isomorphisms \eqref{iso1} and \eqref{iso2} identifying $[1_M]$ with $[x_0]$ are obtained in \cite[Proposition 2.1]{Matthey}. 
		By Poincar\'e duality, we have $H^*(M)\cong H_{3-*}(M)$.
		Moreover, by Kasparov's $K$-theoretic Poincar\'e duality \cite[Theorem 4.10]{Kasparov}\footnote{In order to apply \cite[Theorem 4.10]{Kasparov}, recall that as a closed oriented $3$-manifold, $M$ admits a $\mathrm{spin}^c$-structure and that any $\mathrm{spin}^c$-structure induces a $\Z/2$-graded Morita-equivalence $C_\tau(M)\sim C(M)\hat \otimes  \Cl_3$, where $C_\tau(M)$ denotes the section algebra of the Clifford bundle of $T^*M$ and $\Cl_{3}$ the third complex Clifford algebra.}, we have an isomorphism $K^*(M)\cong K_{3-*}(M)$.
		Together with Bott periodicity, this proves \eqref{iso3} and \eqref{iso4}.
	\end{proof}

	\begin{proof}[Proof of \autoref{thm-intro} {(cf. \cite[Proposition~1.6]{MM})}]

		Using Poincar\'e duality together with $H^1(M)=\Hom_\Z(H_1(M),\mathbb Z)$, the assumptions in \autoref{thm-intro} imply the following isomorphisms:
		\begin{align*}
			H_0(M)&\cong H^3(M)\cong H^0(M)\cong H_3(M)\cong \Z,\\
			H_1(M)&\cong H^2(M)\cong H^1(M)\cong H_2(M)\cong \Z^d.
		\end{align*}
		Combining with \autoref{lma-k-homology}, these allow to compute
		\begin{equation}\label{eq-KofM}
			K_0(M)\cong K_1(M)\cong	K^0(M)\cong	K^1(M)\cong \mathbb Z^{d+1}
		\end{equation}
		and identify 
		\begin{equation}\label{eq-identify}
			K_0(M)  \ni [1_M] \mapsto (1,0)\in \Z\oplus \Z^d.
		\end{equation} 
		
		Note that $M$ has zero Euler characteristic by Poincar\'e duality. Moreover, $G$ satisfies the Baum--Connes conjecture with coefficients by \cite{Lafforgue}. Thus, \cite[Theorem~1]{Gysin} applied to the action $G\curvearrowright \widetilde{M}$ produces short exact sequences
		
		\begin{align}
			\begin{aligned}\label{EM1}
				0\to K_0(C_r^*(G))\xrightarrow{u_*} K_0(C(\partial G)\rtimes_r G) \to K^1(M)\to 0,\\
				0\to K_1(C_r^*(G))\xrightarrow{u_*} K_1(C(\partial G)\rtimes_r G) \to K^0(M)\to 0,
			\end{aligned}
		\end{align}
		
		where the map $u_*$ is induced by the canonical inclusion $u\colon C^*_r(G)\to C(\partial G)\rtimes_r G$.
		Recall that $M$ is a model for $BG$. Thus, the Baum--Connes assembly map 	
		\begin{equation}\label{eq-muG}
			\mu^G_*\colon K_*(M)\xrightarrow{\cong} K_*(C^*_r(G))
		\end{equation}
		is an isomorphism. 
		By applying functoriality of the assembly map to the group homomorphism $\{e\}\to G$, we moreover conclude that
		\begin{equation}\label{eq-mu1}
			\mu^G_0([1_{M}])= [1_{C^*_r(G)}]\in K_0(C^*_r(G)).
		\end{equation}
		Indeed, this follows by applying commutativity of the diagram
		\[ \xymatrix{
			K_0(\pt)\ar[r]^{\mu^{\{e\}}_0} \ar[d] &K_0(\mathbb C)\ar[d]\\
			K_0(M) \ar[r]^{\mu^G_0} &K_0(C^*_r(G))
		}
		\]
		to the element $[1_{\pt}]\in K_0(\pt)$. 
		
		Since $K^*(M)\cong \mathbb Z^{d+1}$, the short exact sequences in \eqref{EM1} split. Using \eqref{eq-muG} at the second step and \eqref{eq-KofM} at the third step, we get
		\[K_0(C(\partial G)\rtimes_r G)\cong K_0(C^*_r(G))\oplus K^1(M)\cong K_0(M)\oplus K^1(M)\cong \Z^{2d+2}.\]

		By \eqref{eq-identify} and \eqref{eq-mu1}, this isomorphism identifies the class of the unit $[1_{C(\partial G)\rtimes_r G}]=u_*([1_{C^*_r(G)}])\in {K_0(C(\partial G)\rtimes_r G)}$ with $(1,0)\in \Z\oplus \Z^{2d+1}$. 
		Analogously, we obtain an isomorphism 
		\[K_1(C(\partial G)\rtimes_r G)\cong K_1(C^*_r(G))\oplus K^0(M)\cong K_1(M)\oplus K^0(M)\cong \Z^{2d+2}.\]		
	\end{proof}

	\begin{proof}[Proof of \autoref{cor-intro}]
		
		By \cite[Theorem 2.1]{brockdunfield}, there are infinitely many pairwise non-isometric closed connected orientable\footnote{Note that although not stated in the theorem, the proof of \cite[Theorem~2.1]{brockdunfield} indeed produces orientable manifolds.} hyperbolic $3$-manifolds $M_i,~ i\in \N$ satisfying $H_1(M_i)\cong \Z^d$. In particular, the fundamental groups $G_i\coloneqq \pi_1(M_i),~ i\in \N$ are pairwise non-isomorphic by Mostow's rigidity theorem \cite{Mostow}. 
		
		The crossed products $A_i\coloneqq C(\partial G_i)\rtimes_r G_i$ are unital simple separable nuclear purely infinite $C^*$-algebras by \cite[Proposition~3.2]{claire} that satisfy the UCT by \cite[Proposition~10.7]{Tu}. 
		Fix $i,j\in \N$. \autoref{thm-intro} implies that there are isomorphisms $K_0(A_i)\cong K_0(A_j)$ and $K_1(A_i)\cong K_1(A_j)$ that identify $[1_{A_i}]$ with $[1_{A_j}]$. Thus, the Kirchberg--Phillips classification theorem \cite{Kirchberg,Phillips} implies that $A_i$ and $A_j$ are isomorphic.
	\end{proof}
	
	\begin{rem}
		We do not know, whether the reduced group $C^*$-algebras $C^*_r(G_i),~i\in \N$ are pairwise isomorphic, although $K$-theory does not distinguish them. To the best of our knowledge, it is an open problem whether the reduced group $C^*$-algebra of a general torsion-free discrete group $G$ recovers the group. However, results of this type do exist for certain classes of groups (see \cite{eckhardt2018c} for instance).
	\end{rem}
	
	As a preparation for the proof of \autoref{prop-intro}, recall that the \emph{transformation groupoid} $X\rtimes G$ associated to a group action $G\curvearrowright X$ is the topological groupoid whose object space is given by $(X\rtimes G)^{(0)}\coloneqq X$ and whose morphism space is given by $X\rtimes G\coloneqq X\times G$, with range, source, and composition maps $r,s,\cdot$ determined by the formulas
	\[r(x,g)=x, \quad s(x,g)=g^{-1}x,\quad (x,g)\cdot (g^{-1}x,h)=(x,gh),\quad x\in X, g,h\in G.\]
	The \emph{topological full group} $[[\mathcal G]]$ of a topological groupoid $\mathcal G$ is the group of all global bisections, that is, open subsets $U\subseteq \mathcal G$ of the arrow space such that the restrictions $r|_U\colon U\to \mathcal G^{(0)}$ and $s|_U\colon U\to \mathcal G^{(0)}$ are homeomorphisms\footnote{Usually, topological full groups are only defined and interesting for groupoids with totally disconnected object space, see \cite{Matui}.}. 
	Note that the topological full group of a topological groupoid is an isomorphism invariant. 
	
	\begin{lma}\label{lma-fullgroup}
		Let $X$ be a connected space and $G\curvearrowright X$ an action of a discrete group. Then the natural map 
		\[G\to [[X\rtimes G]],\quad g\mapsto X\times \{g\}\]
		is a group isomorphism.
	\end{lma}
	\begin{proof}
		If $U\in [[X\rtimes G]]$ is a global bisection, then $U$ is homeomorphic to $X$ and therefore connected. In particular, the projection of $U\subseteq X\times G$ to $G$ is constant. 
	\end{proof}

	\begin{proof}[Proof of \autoref{prop-intro}]
		Assume by contradiction that there is an isomorphism 
		\[C(\partial G_i)\rtimes_r G_i\cong C(\partial G_j)\rtimes_r G_j\]
		that restricts to an isomorphism $C(\partial G_i)\cong C(\partial G_j)$. 
		Note that the action $G_i\curvearrowright \partial G_i$ is topologically free since $G_i$ is torsion-free. 
		Thus, $C(\partial G_i)\subseteq C(\partial G_i)\rtimes_r G_i$ is a Cartan subalgebra (see \cite[\S 6.1]{Renault}). 		
		By Renault's reconstruction theorem \cite[Theorem~5.9]{Renault}, the isomorphism $C(\partial G_i)\rtimes_r G_i\cong C(\partial G_j)\rtimes_r G_j$ is induced by an isomorphism $\partial G_i\rtimes G_i\cong \partial G_j\rtimes G_j$ of the underlying transformation groupoids. In particular, their associated topological full groups are isomorphic. By \autoref{lma-fullgroup}, this implies that $G_i\cong G_j$, which is a contradiction. 
	\end{proof}

	
	\section{$K$-theory with coefficients}
	We end this note with a variant of \autoref{thm-intro} and \autoref{cor-intro} for $3$-manifolds with torsion in their first homology. 
	Recall that for an Abelian group $A$ and a $C^*$-algebra $B$, the $K$-theory of $B$ with coefficients in $A$ is defined as $K_*(B;A)\coloneqq K_*(B\otimes \mathcal A)$ where $\mathcal A$ is any $C^*$-algebra satisfying the UCT such that $K_0(\mathcal A)\cong A$ and $K_1(\mathcal A)=0$. We refer to \cite{RosenbergSchochet} for the UCT and K\"unneth theorem for $C^*$-algebras. 
	\begin{thm}\label{thm-coeff}
		Let $G$ be the fundamental group of a closed connected orientable hyperbolic $3$-manifold. Let $\mathbb F$ be a field. Then there are isomorphisms 
		\[K_0(C(\partial G)\rtimes_r G;\mathbb F )\cong K_1(C(\partial G)\rtimes_r G;\mathbb F )\cong \mathbb F ^2\oplus H_1(M;\mathbb F )\oplus H^1(M;\mathbb F )\]
		which identify the class $[1]\in K_0(C(\partial G)\rtimes_r G;\mathbb F )$ with $(1,0)\in \mathbb F \oplus (\mathbb F \oplus H_1(M;\mathbb F )\oplus H^1(M;\mathbb F ))$.
	\end{thm}
	\begin{proof}
		Since the proof follows the same strategy as the proof of \autoref{thm-intro}, we only explain the necessary modifications.
		Let $\mathcal A$ be a $C^*$-algebra satisfying the UCT such that $K_0(\mathcal A)\cong \mathbb F$ and $K_1(\mathcal A)\cong 0$. 
		
		The $K$-homology of $M$ can be computed using a version of \cite[Proposition~2.1]{Matthey} with coefficients in $\mathbb F$, which in turn can be proven in exactly the same way as the integral version (using the Atiyah--Hirzebruch spectral sequence for $K$-homology with coefficients in $\mathbb F$). 
		Note that the exact sequences \eqref{EM1} are obtained in \cite{Gysin} as long exact sequences in $K$-theory associated to the short exact sequence 
		\[0\to C_0(X)\rtimes_r G \to C(\overline X)\rtimes_r G \to C(\partial G)\rtimes_r G\to 0,\]
		where $X=\widetilde M$ is the universal cover of $M$ and $\overline X$ is its natural compactification. 
		Tensoring this sequence with $\mathcal A$ and applying $K$-theory yields a new six-term exact sequence which again splits into two short exact sequences by the K\"unneth theorem (note that $\mathcal A$ satisfies the K\"unneth theorem since it satisfies the UCT). 
		In this way we obtain the appropriate analogues of the Emerson--Meyer exact sequences with coefficients in $\mathbb F$, 
		\begin{align*}
			0\to K_0(C_r^*(G); \F)\xrightarrow{u_*} K_0(C(\partial G)\rtimes_r G; \F) \to K^1(M; \F)\to 0,\\
			0\to K_1(C_r^*(G); \F)\xrightarrow{u_*} K_1(C(\partial G)\rtimes_r G;\F) \to K^0(M; \F)\to 0,
		\end{align*}
		which split since $\mathbb F$ is a field.
		The left and right hand terms of these sequences are now computed in the same way as before, where Kasparov's Poincar\'e duality theorem \cite[Theorem~4.9]{Kasparov} and the Baum--Connes isomorphism \cite{Lafforgue} have to be taken with coefficients in $\mathcal A$. 
	\end{proof}
	
	\begin{cor}\label{cor-coeff}
		Let $\mathbb F$ be a field and let $\mathcal A$ be a unital simple separable nuclear $C^*$-algebra satisfying the UCT such that $K_0(\mathcal A)\cong \mathbb F$ and $K_1(\mathcal A)\cong 0$. For instance, let $\mathcal A$ be the universal UHF algebra $\mathcal{Q}$ for $\mathbb F = \mathbb Q$, or let $\mathcal A$ be the Cuntz algebra $\mathcal O_{p+1}$ for $\mathbb F=\mathbb F_p$. 
		Let $M$ be a finitely generated Abelian group. Then there are infinitely many pairwise non-isomorphic hyperbolic groups $G_i,~ i\in \N$ with $H_1(G_i)\cong M$ such that the stabilized crossed products 
		\begin{equation*}\label{eq-tensor-product}
			\mathcal A\otimes(C(\partial G_i)\rtimes_r G_i),~ i\in \mathbb N
		\end{equation*}
		are all isomorphic.
	\end{cor}
	\begin{proof}
		Note that the assumptions on $\mathcal A$ guarantee that the stabilized crossed products in the statement are UCT Kirchberg algebras, see \cite[Proposition~4.5]{KirRor} and \cite[Proposition~10.1.7]{BO}. Thus, the corollary can be proven in the same way as \autoref{cor-intro}, using \cite{brockdunfield} and using \autoref{thm-coeff} instead of \autoref{thm-intro}. 
	\end{proof}
	
	\appendix
	\section{Spectral sequences} \label{appendixA}
		
		In this appendix, we give more details on the computation $K$-theory and $K$-homology (with integer coefficients) of a closed, connected, orientable, 3-dimensional, hyperbolic manifold in terms of homology and cohomology, see \cite[Proposition~2.1]{Matthey} and \autoref{lma-k-homology}. This is done using the Atiyah-Hirzebruch spectral sequence.
		
		We recall that a \textit{cohomological spectral sequence} (in the category of Abelian groups) is given by a collection $\{E_i,d_i\}_{i\in\mathbb{N}}$,
		where, for each $i$, the \textit{page} $E_i$ consists of a grid of Abelian groups, denoted $(E_i^{p,q})_{p,q\in\mathbb{Z}}$, equipped with maps $d_i^{p,q}\colon E_i^{p,q}\to E_i^{p+i, q-i+1}$, called \textit{differentials}, which satisfy $d_i^{p,q}\circ d_i^{p-i,q+i-1}=0$, or, in other words, $\mathrm{Im}(d_i^{p-i,q+i-1})\subseteq \mathrm{Ker}(d_i^{p,q})$. Moreover, the ``next page'' $E_{i+1}$ is given by the homology of the page $E_i$. That is, 
		\[E_{i+1}^{p,q}\cong \frac{\mathrm{Ker}(d_i^{p,q})}{\mathrm{Im}(d_i^{p-i,q+i-1})}. \]
		
		The spectral sequence $\{E_i,d_i\}_{i\in\mathbb{N}}$ is called \textit{degenerate at $i_0$} if the differentials $(d_i^{p,q})_{p,q\in\mathbb{Z}}$ vanish for all $i\geq i_0$. In this case, the \textit{limiting term} can be defined as $E_\infty^{p,q}= E_{i}^{p,q}$, for $p,q\in\mathbb{Z}$, and some $i\geq i_0$. We will only deal with degenerate spectral sequences, although the limiting term can be defined more generally.

		Assume that $(A^n)_{n\in\mathbb{Z}}$ is a collection of Abelian groups, so that each $A^n$ admits a Hausdorff grading, which we write as
		\[A^n\coloneqq G^0A^n\supseteq G^1A^n\supseteq G^2A^n\supseteq\ldots \]
		with $\bigcap_{p=0}^{\infty}G^pA^n=0$. It is said that a (degenerate) spectral sequence $\{E_i,d_i\}_{i\in\mathbb{N}}$ \textit{converges} to $(A^n)_{n\in\mathbb{Z}}$ if for every $p\geq 0$ and $q\in\mathbb{Z}$ we have $E_\infty^{p,q}\cong \frac{G^pA^{p+q}}{G^{p+1}A^{p+q}}$. We will use the notation $E_r^{p,q}\Rightarrow A^{p+q}$ for convergence.

		\begin{lma*}
			Let $M$ be a closed connected orientable 3-manifold.
			We have $K^0(M)\cong H^0(M)\oplus H^2(M)$ and $K^1(M)\cong H^1(M)\oplus H^3(M)$.
		\end{lma*}
		
		\begin{proof}
			The second page of the Atiyah-Hirzebruch spectral sequence for $K$-theory (which applies more generally in the setting of finite CW-complexes) is given by 
			\[E_2^{p,q}= H^p(M,K^q(\mathrm{pt})). \]
			We claim that the differentials $d_2^{p,q}$ vanish.
			Indeed, the $K$-theory of a point is
			\[K^q(\mathrm{pt})=\begin{cases}
				\mathbb{Z}, & \text{if q is even}\\
				0, & \text{otherwise}
			\end{cases}. \]
			
			Moreover, under our assumptions on $M$, we have $H^n(M)=0$ for all $n\notin \{0,1,2,3\}$ (this follows, for example, from Poincar\'{e} duality and the fact that cohomology groups are defined to be zero for negative degrees).
			
			Putting these facts together, we can visualize the second page of the Atiyah-Hirzebruch spectral sequence as follows:
			
			
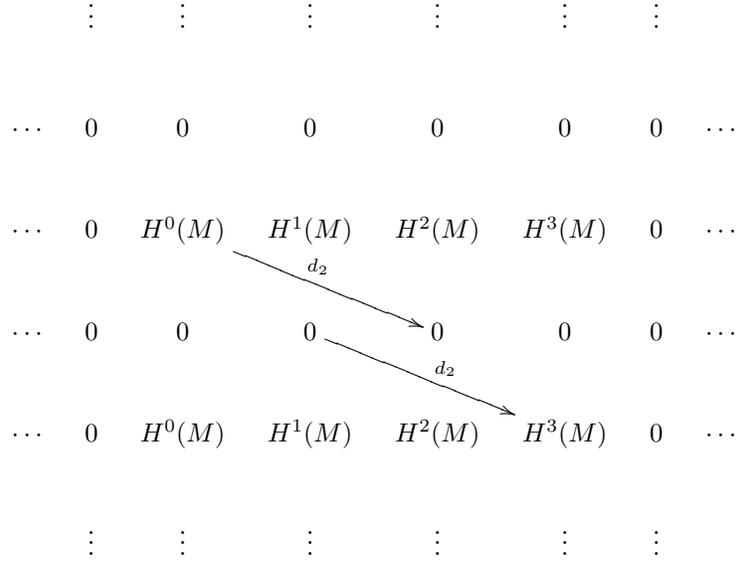
\[\xymatrix@C=1em{
				& \vdots & \vdots & \vdots & \vdots & \vdots & \vdots& \\
				\cdots & 0 & 0 & 0 & 0 & 0 & 0 & \cdots \\
				\cdots & 0 & H^0(M)\ar[drr]^{d_2} & H^1(M) & H^2(M) & H^3(M) & 0 &\cdots \\
				\cdots & 0 & 0&0\ar[drr]^{d_2}&0&0&0 & \cdots \\
				\cdots & 0& H^0(M) & H^1(M) &H^2(M) &H^3(M) & 0 & \cdots\\
				& \vdots & \vdots & \vdots & \vdots & \vdots & \vdots &
			}\]
			\mbox{}\\

			It is immediately visible that the differentials $d_2$ either have vanishing range or vanishing domain, thus must all be the zero map.
			The third page in the Atiyah-Hirzebruch spectral sequence is now concluded to be identical (as a grid) to the second page (since $E_3$ is defined to be the homology of $E_2$). We will check that the differentials $d_3$ vanish as well. In the following figure we draw the only (recurring) differentials which could potentially be non-zero:
			
			\[\xymatrix@C=1em{
				& \vdots & \vdots & \vdots & \vdots & \vdots & \vdots& \\
				\cdots & 0 & 0 & 0 & 0 & 0 & 0 & \cdots \\
				\cdots & 0 & H^0(M)\ar[ddrrr]^{d_3} & H^1(M) & H^2(M) & H^3(M) & 0 &\cdots \\
				\cdots & 0 & 0&0&0&0&0 & \cdots \\
				\cdots & 0& H^0(M) & H^1(M) &H^2(M) &H^3(M) & 0 & \cdots\\
				& \vdots & \vdots & \vdots & \vdots & \vdots & \vdots &
			}\]
			\mbox{}\\
			
			Applying the Atiyah-Hirzebruch spectral sequence to a one-point space, we obtain exactly the same second and third page as obtained above for the manifold $M$, except $H^n(\mathrm{pt})=0$ for all $n\neq 0$.
			
			Consider the natural maps $\mathrm{pt}\to M\to\mathrm{pt}$. Since the Atiyah-Hirzebruch spectral sequence is functorial, we have a retract $E_3(\mathrm{pt})\to E_3(M)\to E_3(\mathrm{pt})$, which is interpreted as the following commutative diagram.
			\[\xymatrix{
				H^0(\mathrm{pt})\ar[d]^{d_3}\ar[r] & H^0(M)\ar[d]^{d_3}\ar[r] & H^0(\mathrm{pt})\ar[d]^{d_3} \\
				H^3(\mathrm{pt})\ar[r] & H^3(M)\ar[r] & H^3(\mathrm{pt})
			}\]
			
			As $H^3(\mathrm{pt})=0$, we have that the differential maps on the left-hand side and right-hand side must vanish. The composition $\xymatrix@C=1.5em{H^0(\mathrm{pt})\ar[r] & H^0(M)\ar[r] & H^0(\mathrm{pt})}$ is the identity map. By our assumptions on $M$ (and Poincar\'{e} duality), we have $H^0(M)\cong H^3(M)\cong H^0(\mathrm{pt})\cong \mathbb{Z}$, which forces the maps $\xymatrix@C=1.5em{H^0(\mathrm{pt})\ar[r] & H^0(M)}$ and $\xymatrix@C=1.5em{H^0(M) \ar[r] & H^0(\mathrm{pt})}$ to be isomorphisms. Rewriting the diagram above, we have
			
			\[\xymatrix{
				\Z\ar[d]^{d_3}\ar[r]^{\cong} & \Z\ar[d]^{d_3}\ar[r]^{\cong} & \Z\ar[d]^{d_3} \\
				0\ar[r] & \Z\ar[r] & 0
			}\]
			which now clearly forces the map $d_3\colon H^{0}(M)\to H^3(M)$ to be the zero map, as we wanted to show.
			
			Observe that for $n\geq 3$, $E_n$ is identical to $E_2$ as a grid, and it is easy to see that the differentials $d_n$ vanish as well for all $n\geq 4$. We conclude that the Atiyah-Hirzebruch spectral sequence is degenerate at $i_0=2$, and so $E_\infty^{p,q}=E_2^{p,q}$, for all $p,q\in \mathbb{Z}$.
			
			Set $A^n\coloneqq K^n(M)$, for $n\in\mathbb{Z}$. By Atiyah--Hirzebruch, there exists a Hausdorff grading $A^n=G^0A^n\supseteq G^1A^n\supseteq G^3A^n\supseteq \ldots$, for every $n\in\mathbb{Z}$ (the concrete description of the grading will follow from the computations below), such that $E_2^{p,q}\Rightarrow K^{p+q}(M)$. 
			
			For $p,q\in\mathbb{Z}$ with $p+q=0$ we get the following short exact sequences:
			\[\xymatrix{
				0\ar[r]& G^1A^0\ar[r] & G^0A^0\ar[r] & E_{\infty}^{0,0}\ar[r] &0 \\
				0\ar[r]& G^2A^0\ar[r] & G^1A^0\ar[r] & E_{\infty}^{1,-1}\ar[r] &0 \\
				0\ar[r]& G^3A^0\ar[r] & G^2A^0\ar[r] & E_{\infty}^{2,-2}\ar[r] &0\\
				0\ar[r]& G^4A^0\ar[r] & G^3A^0\ar[r] & E_{\infty}^{3,-3}\ar[r] &0\\
				& & \vdots
			}\]
			
			\mbox{}\\
			The first short exact sequence reduces to 
			\[\xymatrix{0\ar[r]& G^1A^0\ar[r]& K^0(M)\ar[r]& H^0(M)\ar[r]& 0}. \]
			Since $H^0(M)\cong \mathbb{Z}$ is torsion free, we conclude that $K^0(M)\cong H^0(M)\oplus G^1A^0$.
			
			If $n$ is odd, we have $E_\infty^{n,-n}=0$, so $G^{1}A_0\cong G^{2}A^0$ and $G^3A^0\cong G^4A^0$.
			Note that $E_\infty^{n,-n}=0$ whenever $n\geq 4$. Thus, $G^nA^0\cong G^3A^0$ for all $n\geq 4$. Since the grading is Hausdorff, it follows that $G^nA^0=0$ for all $n\geq 3$.
			Finally, $E_\infty^{2,-2}\cong H^2(M)$.
			Working inductively with the short exact sequences above, we conclude that 
			\[K^0(M)\cong H^0(M)\oplus G^1A^0\cong H^0(M)\oplus G^2A^0\cong H^0(M)\oplus H^2(M). \]
			
			Next, we compute $K^1(M)$ using $E_2^{p,q}\Rightarrow K^{p+q}(M)$. For $p,q\in\mathbb{Z}$ with $p+q=1$ we get the following short exact sequences
			
			\[\xymatrix{
				0\ar[r]& G^1A^1\ar[r] & G^0A^1\ar[r] & E_{\infty}^{0,1}\ar[r] &0 \\
				0\ar[r]& G^2A^1\ar[r] & G^1A^1\ar[r] & E_{\infty}^{1,0}\ar[r] &0 \\
				0\ar[r]& G^3A^1\ar[r] & G^2A^1\ar[r] & E_{\infty}^{2,-1}\ar[r] &0\\
				0\ar[r]& G^4A^1\ar[r] & G^3A^1\ar[r] & E_{\infty}^{3,-2}\ar[r] &0\\
				& & \vdots
			}\]
			
			\mbox{}\\
			Since $E_\infty^{0,1}\cong E_\infty^{2,-1}\cong 0$, the first short exact sequence implies that $K^1(M)\cong G^1A^1$, and the third short-exact sequence implies $G^2A^1\cong G^3A^1$.
			Since $E_\infty^{1,0}\cong H^1(M)\cong \mathrm{Hom}(M,\Z)$ is torsion-free and finitely generated, we have that the second short exact sequence splits, and $G^1A^1\cong H^1(M)\oplus G^2A^1\cong H^1(M)\oplus G^3A^1$. 
			
			Finally, $E_\infty^{3,-2}\cong H^3(M)\cong H_0(M)\cong \Z$ by Poincar\'{e} duality. Thus, the last short exact sequence splits as well, and we have $G^3A^1\cong H^3(M)\oplus G^4A^1$. However, noticing that $E_\infty^{n,m}=0$ for all $n\geq 4$ and all $m\in\mathbb{Z}$, we see that $G^nA^1\cong G^4A^1$ for all $n\geq 4$. Since the grading is Hausdorff, it follows that $G^4A^1=0$, and so $G^3A^1\cong H^3(M)$. Combining the above,
			\[K^1(M)\cong G^1A^1\cong H^1(M)\oplus G^2A^1\cong H^1(M)\oplus G^3A^1\cong H^1(M)\oplus H^3(M).\]
			
		\end{proof}
		
		In a similar manner, one can apply the homological version of the Atiyah-Hirzebruch spectral sequence in order to compute the $K$-homology of $M$, $K_*(M)$. Alternatively, one can use Kasparov's $K$-theoretic Poincar\'e duality and Bott periodicity (as explained in the proof of \autoref{lma-k-homology}) to conclude that $K_0(M)\cong H_0(M)\oplus H_2(M)$ and $K_1(M)\cong H_1(M)\oplus H_3(M)$.

	\bibliography{refs}{}
	\bibliographystyle{alpha}
\end{document}